\documentclass[a4paper]{amsart}
\usepackage[english]{babel}
\usepackage[top=2cm,bottom=2cm,left=3cm,right=3cm,marginparwidth=1.75cm]{geometry}

\usepackage{biblatex}
\renewbibmacro{in:}{\addcomma\addspace}

\usepackage{amsaddr,amsmath, amsfonts, amssymb,graphicx, tikz-cd,tikz,enumitem,mathtools,hyperref, mathabx, mathrsfs, mathabx, cleveref, mdframed, lipsum,csquotes}
\usetikzlibrary{calc,shapes}

\theoremstyle{plain}
\newtheorem{theorem}{Theorem}
\newtheorem{lemma}[theorem]{Lemma}
\newtheorem{corollary}[theorem]{Corollary}
\newtheorem{proposition}[theorem]{Proposition}

\theoremstyle{definition}
\newtheorem{definition}[theorem]{Definition}
\newtheorem*{theorem*}{Theorem}
\newtheorem*{corollary*}{Corollary}
\newtheorem{example}[theorem]{Example}

\theoremstyle{remark}
\newtheorem{remark}[theorem]{Remark}

\newcommand{\Q}{\mathbb{Q}}

\begin{document}

\title[Nil-Essential Ideals]{Nil-Essential Ideals}
\author[Nongsiej, Buhphang]{Raplang Nongsiej, Ardeline Mary Buhphang}
\address{
Department of Mathematics, North Eastern Hill University, Shillong, India\\ Corresponding author: Ardeline Mary Buhphang}
\email{ardeline@nehu.ac.in, rap890002@gmail.com}












\begin{abstract}
The class of nil-essential ideals is a generalisation of the class of essential ideals. Every nil-essential ideal of a reduced ring is essential. Therefore the intersection of all nil-essential ideals over a reduced ring $R$ is the socle of $R$. In this note, we apply this generalisation to give a new criteria of semisimplicity in terms of nil-essentiality of ideals. 
\end{abstract}

\maketitle

\section{Introduction}

Throughout, we assume a ring to be an associative ring with identity, unless mentioned otherwise. Essential submodules were first studied by Johnson \cite{john} in the year 1951 and its nomenclature is credited to Eckmann and Schopf \cite{eck}. Essential ideals play a very important role in the field of commutative and noncommutative algebras. For example, $socle(_RR)=\bigcap \left\{I\mid I \text{ is essential left ideal}\right\}$. If $I$ is a left(resp. right) ideal of a ring $R$ and $J\leq I$ be another left(resp. right) ideal then $J$ is said to be \emph{essential(or large)} \cite{john} in $I$, written as $J\unlhd I$ \cite[p.~72]{ander} in case for any left ideal $\mu\leq I$, whenever $J\cap \mu=0$, we have $\mu=0$. By replacing $\mu=0$ with $\mu$ being nilpotent we introduce the notion of \emph{nil-essential ideals}, which is a generalised concept of essential ideals. Our focus in this paper is on ideals of a ring since it deals with the nilpotency property of subsets of a ring. A left(resp. right) ideal $I$ of a ring $R$ (not necessarily commutative) is said to be nil-essential if whenever $I \cap\mu = 0$ for some left(resp. right) ideal $\mu$ of $R$ we have $\mu$ is nilpotent. We shall adopt the notation $I \unlhd_{nil} R $ to mean that $I$ is a nil-essential ideal of $R$.

All essential ideals are nil-essential, however the converse is not true which will be shown later with an example. In this paper we investigate the properties of nil-essential ideals and include results on them and their localisations. While most of the results are for general rings, few of them are restricted to the case when $R$ is a commutative Noetherian ring.

A ring $R$ is called a reduced ring if it has no non-zero nilpotent elements. So, $0$ is the only nilpotent ideal in a reduced ring; naturally it follows that the essential and nil-essential left ideals coincide in such a ring. Therefore in a reduced ring $R$, the intersection of all nil-essential ideals is indeed the socle of $R$.

Section 2 of this paper deals with basic properties of nil-essential ideals. In Section 3, we include the concept of nil-essential monomorphisms. While in the final section, we investigate the localisation of nil-essential ideals.

\section{Nil-Essential Ideals}
We start this section by recording some basic properties of nil-essential ideals.

\begin{proposition}\label{201}
	Let $R$ be a ring. If a left(resp. right) ideal $I$ is nil-essential, then every left(resp. right) ideal of $R$ containing $I$ is nil-essential.
\end{proposition}

\begin{corollary}\label{202}
	Let $R$ be a ring. Let $I_1,I_2,\dots,I_n$ be left(resp. right) ideals of a ring R. Then whenever $ I_1 \cap I_2 \cap\dots\cap I_n $  is nil-essential we have $I_j$ is nil-essential for each $j\in \{1,2,\dots,n\}$. 
\end{corollary}
	%
 Converse to the above corollary does not hold in general, for instance consider the ring \\
                $$R= \left\lbrace \left[\begin{array}{lccr}
			a & b & c \\
			0 & a & d \\
			0 & 0 & a
			\end{array}
			\right]: a,b,c,d \in \mathbb Q 
			\right\rbrace$$
			
			then
			$I =  \left \lbrace \left[\begin{array}{lcr}
			0 & b & 0 \\
			0 & 0 & 0 \\
			0 & 0 & 0
			\end{array}
			\right]:b\in \mathbb Q \right\rbrace \quad \text{ and } \quad J =  \left \lbrace \left[\begin{array}{lcr}
			0 & 0 & c \\
			0 & 0 & 0 \\
			0 & 0 & 0
			\end{array}
			\right]:c\in \mathbb Q \right\rbrace \, $ are left ideals of $R$ which are both nil-essential whereas their intersection is not.

   It is to be noted that in the above example we can replace $\Q$ by any ring that contains identity or a non-nilpotent element.

\begin{corollary}\label{203}
	If $ I_1,I_2,\dots,I_n $ are (left)ideals of a ring $R$ such that $I_k $ is nil-essential for some $1\leq k\leq n$. Then $I_1 + I_2 +\dots+ I_n$ is nil-essential.
\end{corollary}

\begin{corollary}\label{204}
	If $ I_1,I_2,\dots,I_n $ are (left)ideals of a ring $R$ such that $\prod_{i=1}^{n}I_i$ is nil-essential. Then $I_i$ is nil-essential for every $1\leq i\leq n$.
\end{corollary}
\begin{corollary}\label{205}
	Let I be an (left)ideal of a ring $R$ such that $I^n$ is nil-essential for some $n\geq1$, then $I$ is nil-essential ideal. 
\end{corollary}
Converse to the above assertion is not always true. For example if, $$R= \left\lbrace \left[\begin{array}{lccr}
		a & b & c \\
		0 & a & d \\
		0 & 0 & a
		\end{array}
		\right]: a,b,c,d \in \mathbb Q 
		\right\rbrace  \quad \text{ and }\quad I = \left \lbrace \left[\begin{array}{lcr}
		0 & 0 & c \\
		0 & 0 & 0 \\
		0 & 0 & 0
		\end{array}
		\right]:c\in \mathbb Q \right\rbrace,$$ then $I$ is nil-essential but $I^2=0$ is not.

\begin{proposition}\label{206}
	Let $I\subseteq J\subseteq K \text{ and }L$ be (left)ideals of a ring R. Then
	\begin{enumerate}
		\item $I\unlhd_{nil} J$ and $J\unlhd_{nil} K$ if $I\unlhd_{nil} K$ .
		\item $I\unlhd_{nil}K$ and $L\unlhd_{nil}K$ if $(I\cap L)\unlhd_{nil}K$.
	\end{enumerate}
\end{proposition}
\begin{proof}
	The proof is similar to the case of essential ideals in \cite{goodearl}.
\end{proof}

\begin{proposition}\label{207}
	Let $R$ be a ring. Let $M$ be a maximal ideal and $\mu$ an ideal of $R$ such that $\mu\cap M=0$, then either $\mu=0$ or $\mu$ is non-nilpotent.
\end{proposition}
\begin{proof}
	Let $\mu\cap M=0$. If $\mu\neq0$ then there exists $x\in \mu$ such that $x\notin M$. Which follows that $M + Rx = R$ and $1\in R$ and therefore we have $m+rx=1$ for some $r\in R$ and $m\in M$. Hence, $rx$ cannot be nilpotent, otherwise $m\in M$ would be a unit. Therefore, $\mu$ is non-nilpotent.
\end{proof}

The following corollary will provide us with a necessary condition for a maximal ideal of a ring $R$ to be nil-essential. 
\begin{corollary}\label{208}
	For a maximal ideal $M$ of a ring $R$, $M$ is nil essential implies that $M$ intersects any non-zero ideal non-trivially. 
\end{corollary}

\noindent We recall that an $R$-module $M$ is semi-simple if and only if every submodule of $M$ is a direct summand of $M$ \cite{goodearl}.

\begin{corollary}\label{209}
	Let $R$ be a ring. Then the following statements are equivalent:
	\begin{enumerate}
		\item Every proper ideals are non nil-essential.
		\item Every maximal ideals are non nil-essential.
		\item $R$ is a semi-simple ring.
	\end{enumerate}
\end{corollary}

\begin{corollary}\label{210}
	Let $R$ be a ring. Let $J$ be the Jacobson `ical of $R$ and $\mu$ be an ideal such that $\mu\cap  J=0$. Then either $\mu=0$ or $\mu$ is non-nilpotent.
\end{corollary}
\begin{proof}
	Similar to Proposition \ref{207}.
\end{proof}

\begin{remark}\label{211}
\hfill
 \begin{enumerate}
     \item Corollary \ref{208} holds if we replace $M$ by $J$, the Jacobson radical of the given ring. In a semi-simple ring, the Jacobson radical (or nil - radical) and any maximal ideal cannot be nil-essential. Indeed, in a semi-simple ring both the Jacobson radical as well as the nil-radical are zero. Consequently, in a commutative semi-simple ring we cannot obtain a nil-essential prime ideal. 
     \item Corollary \ref{208} fails to hold when the Jacobson radical of the ring is zero. For instance, if $R=\mathbb Z$ every non-zero ideal of $R$ is nil-essential.
 \end{enumerate}

\end{remark}

The following result will provide a necessary and sufficient criterion for a non-zero ideal of a Noetherian ring $R$ to be nil-essential. 

\begin{lemma}\label{212}
	Let $R$ be a commutative Noetherian ring and $I$ a non-zero ideal. Then
	 $I$ is a nil-essential ideal if and only if for each $x\in R$ with $x$ non-nilpotent there exists $r\in R$ such that $rx\in I$ and $rx\neq0$.
	
\end{lemma}
\begin{proof}
	Suppose $x\in R$ is non-nilpotent, then $I\cap Rx\neq0$, since $I$ is nil-essential, there exists $r\in R$ such that $rx\in I$ and $rx\neq0$. 
 
 Conversely, assume $\mu$ to be a non-nilpotent ideal of $R$ such that $I\cap\mu=0$. Then, there exists $x\in \mu$ such that $x$ is not nilpotent (since $R$ is noetherian \cite{atiyah}). By hypothesis, there exists $r\in R$ such that $rx\neq0$ and $rx\in I$. But, this yields that $rx \in I \cap \mu$, a contradiction. Therefore, $\mu$ is nilpotent and so $I$ is nil-essential. 
\end{proof}

\begin{lemma}\label{213}
	Let $R$ be a commutative noetherian ring and $I\subseteq J$ be non-zero ideals. Then, $I\unlhd_{nil}J$ if and only if for each $x\in J$ with $x$ is non-nilpotent there exists $r\in R$ such that $rx\in I$ and $rx\neq0$.
	
\end{lemma}
\begin{proof}
	Similar to Lemma \ref{212}.
\end{proof}

\begin{proposition}\label{214}
	Let $R$ be a commutative noetherian ring and $I$ an ideal. Then $I\unlhd_{nil}rad(I)$, where $rad(I)$ is the radical of $I$.
\end{proposition}
\begin{proof}
	Follows easily from definition of $rad(I)$ and Lemma \ref{213}.
\end{proof}

\begin{proposition}\label{215}
	Let $R$ be a noetherian ring. Let $I_1, I_2, J_1, J_2$ be ideals of $R$ such that $I_1 \subseteq J_1$, $I_2\subseteq J_2$ and $J_1\cap J_2=0$. Then the following are equivalent:
	\begin{enumerate}
		\item $I_1\unlhd_{nil}J_1$ and $I_2\unlhd_{nil}J_2$.
		\item $I_1\oplus I_2\unlhd_{nil}J_1\oplus J_2$.
	\end{enumerate}
\end{proposition}
\begin{proof}
	$(2) \Rightarrow (1)$: Let $\mu$ be an ideal of $R$ and $\mu\subseteq J_1$ such that $I_1\cap \mu=0$, then $I_1\oplus I_2\cap \mu=0$. Therefore $\mu$ is nilpotent (since $I_1\oplus I_2\unlhd_{nil}J_1\oplus J_2$ and $\mu\leq J_1\leq J_1\oplus J_2$). Hence, $I_1\unlhd_{nil}J_1$. Similarly, $I_2\unlhd_{nil}J_2$.
	
	\qquad$(1) \Rightarrow (2)$: Let $x\in (J_1\oplus J_2)$ be non-nilpotent. Then  $x=x_1+x_2$ is non-nilpotent, for some $x_1\in J_1$ and $x_2\in J_2$. Then we have $x_1$ or $x_2$ is non-nilpotent. Without any loss, $x_1$ is non-nilpotent, then $x_1x=x_1(x_1+x_2)=x_1^2\neq0$. As $x_1$ is non-nilpotent so is $x_1^2$. Therefore since $I_1\unlhd_{nil}J_1$, there exists $r\in R$ such that $rx_1^2\in I_1\subseteq I_1\oplus I_2$ and $rx_1^2\neq0$. Hence by Lemma \ref{213}, $I_1\oplus I_2\unlhd_{nil}J_1\oplus J_2$.
\end{proof}

Recall that, for an ideal $I$ of a ring $R$ and for $a\in R$, the ideal quotient $(I:a)$ is defined as $(I:a):=\{r\in R: ra\in I\}$.

\begin{proposition}\label{216}
	Let $I$ and $J$ be ideals over a commutative noetherian ring $R$ such that $I\unlhd_{nil}J$ then $(I:a)\unlhd_{nil}R$ for every $a\in J$.
\end{proposition}
\begin{proof}
	By Lemma \ref{213}, in order to show that $(I:a)\unlhd_{nil}R$, we have to show that for every $x\in R$, with $x$ non-nilpotent there exists an $r\in R$ such that $rx\neq0$ and $rx\in (I:a)$.
	
	We have the following cases for any non-nilpotent element $x\in R$.
	
	Case(i): If $xa$ is non-nilpotent. Since $xa\in J$ and $I\unlhd_{nil}J$, by Lemma \ref{213}, there exists $r\in R$ such that $rxa\neq0$ and $rxa\in I$. Therefore $rx\neq 0$ and $rx\in (I:a)$.

\vspace{2mm}
	Case(ii): If $xa$ is nilpotent, we have $(xa)^n=0$ for some $n\in \mathbb N$, so $(xa)^n\in I$ and therefore we get that $x^na^{n-1}\in (I:a)$. If $x^na^{n-1}\neq0$, we take $r=x^{n-1}a^{n-1}$ and therefore $rx\neq0$ and $rx\in(I:a)$. If not, i.e., if $x^na^{n-1}=0$, we have $x^na^{n-2}\in (I:a)$. Again if $x^na^{n-2}\neq0$, we take $r=x^{n-1}a^{n-2}$ and get $rx\neq0$ and $rx\in (I:a)$. But if $x^na^{n-2}=0$, proceeding the same way and also using the fact that $x$ is non-nilpotent, there exists $k\in \{1,2,\dots,n\}$ such that $x^na^{n-k}\neq0$ and $x^na^{n-k+1}=0$. Now taking $r=x^{n-1}a^{n-k}$ we get $rx\neq0$ and $rx\in(I:a)$.
	
	Hence we see that in both cases, whenever $x\in R$ is non-nilpotent, there exists an element $r\in R$ such that $rx\neq0$ and $rx\in (I:a)$.
\end{proof}

\section{Nil-Essential Monomorphisms}
In this section we will investigate the notion of nil-essential ideals from the perspective of monomorphisms. We begin this section by introducing the notion of \emph{nil-essential monomorphism} and subsequently list out their basic properties. 

\begin{definition}\label{217}
	Let $I$ and $J$ be ideals over a ring $R$. A monomorphism $f:I\longrightarrow J$ is {\it nil-essential} if $Im(f) \unlhd_{nil} J$.
\end{definition}

The following proposition characterises nil-essential ideals in terms of nil-essential monomorphisms. 
\begin{proposition}\label{218}
	Let $J$ be an ideal over a ring $R$ and $I$ an ideal of $R$ contained in $J$. Then the following statements are equivalent:
	\begin{enumerate}
		\item $I \unlhd_{nil} J$
		\item The inclusion map $i:I\longrightarrow J$ is a nil-essential monomorphism.
		\item For each ideal $K$ of $R$ and for all $f\in Hom(J,K)$ whenever, $ker(f) \cap I=0$, we have $ker(f)$ is nilpotent.
	\end{enumerate}
\end{proposition}

\begin{proposition}\label{219}
	Let $f:R\longrightarrow R$ be ring homomorphism. If an ideal $J$ is nil-essential, then $f^{-1}(J)$ is nil-essential. 
\end{proposition}

\begin{remark}\label{220}
	\noindent{In the example below we will illustrate that essential and non-essential ideals do not coincide in general.} Consider, 
	$$R= \left\lbrace \left[\begin{array}{lccr}
	a & b & c \\
	0 & a & d \\
	0 & 0 & a
	\end{array}
	\right]: a,b,c,d \in \mathbb Q 
	\right\rbrace,$$then ideals of $R$ are:
	$I_1=R$,
	$I_2=0$,\\
        $I_3= \left\lbrace \left[\begin{array}{lccr}
	0 & 0 & c \\
	0 & 0 & 0 \\
	0 & 0 & 0
	\end{array}
	\right]: c \in \mathbb Q 
	\right\rbrace$, \qquad 
	$I_4= \left\lbrace \left[\begin{array}{lccr}
	0 & b & 0 \\
	0 & 0 & 0 \\
	0 & 0 & 0
	\end{array}
	\right]: b \in \mathbb Q 
	\right\rbrace$,\\
	$I_5= \left\lbrace \left[\begin{array}{lccr}
	0 & b & c \\
	0 & 0 & 0 \\
	0 & 0 & 0
	\end{array}
	\right]: b,c \in \mathbb Q 
	\right\rbrace$,\qquad
	$I_6= \left\lbrace \left[\begin{array}{lccr}
	0 & 0 & c \\
	0 & 0 & d \\
	0 & 0 & 0
	\end{array}
	\right]: c,d \in \mathbb Q 
	\right\rbrace$,\\
	$I_7= \left\lbrace \left[\begin{array}{lccr}
	0 & b & b \\
	0 & 0 & 0 \\
	0 & 0 & 0
	\end{array}
	\right]: b \in \mathbb Q 
	\right\rbrace$,\qquad 
	$I_8= \left\lbrace \left[\begin{array}{lccr}
	0 & b & c \\
	0 & 0 & d \\
	0 & 0 & 0
	\end{array}
	\right]: b,c,d \in \mathbb Q 
	\right\rbrace$,\\
	$I_9= \left\lbrace \left[\begin{array}{lccr}
	0 & b & c \\
	0 & 0 & b \\
	0 & 0 & 0
	\end{array}
	\right]: b,c \in \mathbb Q 
	\right\rbrace.$\\

	Here $I_3$ is nil-essential but not essential, since $I_3\cap \mu=0$ yields either $\mu=0,I_4,$ or $I_7$ and these are all nilpotent ideals. Similarly, we have  $I_4$ is nil-essential but not essential.	
\end{remark}

\section{Localisation of Nil-Essential Ideals}

In this section we present some basic theories on localisation of nil-essential ideals.

\begin{proposition}\label{222}
	Let $R$ be a commutative noetherian ring, $S$ be the set of non-zero divisors of $R$ and $I$ be an ideal of $R$. Then $I$ is nil-essential iff $S^{-1}I$ is nil-essential.
\end{proposition}
\begin{proof}
	($\Rightarrow$): Let $\dfrac{x}{s}\in S^{-1}R$ be non-nilpotent, i.e., $\left(\dfrac{x}{s}\right) ^n \neq\dfrac{0}{1}$,  $\forall \; n \in \mathbb N$, therefore $x^n \neq 0$, $\forall \; n \in \mathbb N$, i.e., $x$ is non-nilpotent. By assumption, there exists $\ r\in R$ such that $rx\in I$ and $rx\neq0$. Then $\dfrac{r}{1}\in S^{-1}R$ and $\dfrac{r}{1} \dfrac{x}{s}=\dfrac{rx}{s}\neq \dfrac{0}{1}$ in $S^{-1}I$ and so by Lemma \ref{213}, $S^{-1}I$ is nil essential.\\
	($\Leftarrow$): Suppose $S^{-1}I$ is nil-essential.\\
	Let $x\in R$ be non-nilpotent so $\dfrac{x}{1}$ is non-nilpotent in $S^{-1}R$. Therefore, there exists $\dfrac{r}{s} \in S^{-1}R$ such that, $\dfrac{r}{s}\dfrac{x}{1}\neq \dfrac{0}{1}$ and $\dfrac{r}{s}\dfrac{x}{1}\in S^{-1}I$. Since $\dfrac{rx}{s}\in S^{-1}I$ we have $\dfrac{rx}{s}=\dfrac{a}{t}$ for some $a\in I$ and $t\in S$ therefore $rxtu=asu$ for some $u\in S$ also $(rtu)x=asu\in I ($as $a\in I$) and $(rtu)x\neq 0$ (For if $rtux=0$ then $\dfrac{a}{t}=\dfrac{0}{1}$ which is a contradiction). Therefore $I$ is nil-essential by Lemma \ref{213}.
\end{proof}

\begin{corollary}\label{223}
	Let $R$ be a commutative, noetherian ring, $S$ be a multiplicatively closed subset of $R$ and $I$ be an ideal of $R$. Then $I$ is nil-essential if $S^{-1}I$ is nil-essential.
\end{corollary}

It is to be noted that the above corollary does not hold in general. Following this we will document an example of a ring whose localisations of nil-essential ideals are again nil-essential.

\begin{example}\label{224}
	Let $I$ be an ideal of a ring ${\mathbb Z}_{p^n}$ where $p$ is prime and $n\in N$. Then if $S$ is a multiplicatively closed subset of ${\mathbb Z}_{p^n}$ we have $I$ is essential iff $S^{-1}I$ is nil-essential.
\end{example}

\section*{acknowledgment}
The second author would like to thank Mr. Rishabh Goswami (North-Eastern Hill University, Shillong) for carefully reviewing the manuscript and bring it to the present form.

\section*{Declarations}
\subsection*{Ethical approval}
Not applicable

\subsection*{Funding}
Not applicable

\subsection*{Disclosure} The authors report that there are no competing interests to declare

\subsection*{Authors' contribution}
All authors have contributed equally to all sections.

\subsection*{Data availability statement}
Not applicable









\printbibliography
\end{document}